\documentclass[a4paper,10pt]{amsart}

\usepackage{graphicx}
\usepackage[dvips]{color}
\usepackage[all]{xy}
\usepackage[english]{babel}

\usepackage{amsmath,amssymb}

\usepackage{amsthm}

\usepackage{amscd}

\usepackage{mathrsfs}
\usepackage{epsfig}

\usepackage{stmaryrd}

\DeclareMathOperator{\Pic}{Pic}

\DeclareMathOperator{\spec}{Spec}

\DeclareMathOperator{\Hilb}{Hilb}

\begin{document}

 \title[Compactified Picard stacks over $\overline{\mathcal M}_{g,n}$]{Compactified Picard stacks over the moduli stack of stable curves with marked points}
\author{Margarida Melo}\thanks{Partially supported by a Funda\c{c}\~ao Calouste
Gulbenkian fellowship.}

\newcommand{\sm}{\setminus}
\newcommand{\av}{``}
\newcommand{\uv}{''}
\newcommand{\forn}{\forall n\in\mathbb{N}}
\newcommand{\raw}{\to}
\newcommand{\Raw}{\Rightarrow}
\newcommand{\law}{\gets}
\newcommand{\Law}{\Leftarrow}
\newcommand{\Lra}{\Leftrightarrow}
\newcommand{\C}{\mathcal C}
\newcommand{\K}{\mathcal K}
\newcommand{\A}{\mathcal A}
\newcommand{\N}{\mathbb N}
\newcommand{\X}{\mathcal X}
\newcommand{\M}{\mathcal M}
\newcommand{\G}{\mathcal G}
\newcommand{\Z}{\mathbb Z}
\renewcommand{\P}{\mathcal P}
\renewcommand{\o}{\overline}

\def\hifen{\discretionary{-}{-}{-}}

\newtheorem{teo}{Theorem}[section]
\newtheorem{prop}[teo]{Proposition}
\newtheorem{lem}[teo]{Lemma}
\newtheorem{rem}[teo]{Remark}
\newtheorem{defi}[teo]{Definition}
\newtheorem{ex}[teo]{Example}
\newtheorem{exs}[teo]{Examples}
\newtheorem{cor}[teo]{Corollary}

\begin{abstract}
In this paper we give a construction of algebraic (Artin) stacks $\o\P_{d,g,n}$ endowed with a modular map onto the moduli stack of pointed stable curves $\o{\mathcal M}_{g,n}$, for $g\geq 3$. The stacks $\o\P_{d,g,n}$ are smooth, irreducible and have dimension $4g-3+n$. They yield a geometrically meaningful compactification of the degree $d$ universal Picard stack over $\mathcal M_{g,n}$, parametrizing $n$-pointed smooth curves together with a degree $d$ line bundle.
\end{abstract}

\maketitle

\let\languagename\relax

\section{Introduction}
Consider the degree $d$ universal Picard stack over $n$-pointed smooth curves of genus $g$, $\mathcal Pic_{d,g,n}$, parametrizing families of $n$-pointed smooth curves of genus $g$ endowed with line bundles of relative degree $d$ over these curves.

$\mathcal Pic_{d,g,n}$ is, of course, not complete.
In the present paper we search for a compactification of $\mathcal Pic_{d,g,n}$ over $\o{\mathcal M}_{g,n}$, the moduli stack of stable curves of genus $g$. By this we mean an algebraic stack with a proper (or at least universally closed) map onto $\o{\mathcal M}_{g,n}$, containing $\mathcal Pic_{d,g,n}$ as a dense open substack

The problem of compactifying the (generalized) jacobian of curves or of families of curves has been widely studied in the last decades, starting from the work of Igusa \cite{igusa} and of Mayer and Mumford \cite{meyermumf}. The first solution to this problem, in the case on irreducible curves, was due to D'Souza in \cite{dsouza} and since then, several other solutions where found, the more general one being probably Simpson's construction of moduli spaces of coherent sheaves on
projective schemes in \cite{simpson} (see \cite{alex} for an overview and comparison results on these constructions). 

On the other hand, the construction of the moduli space of stable curves with marked points was done by Knudsen in \cite{knudsen}, following ideas of Mumford, with the scope of proving the projectivity of the moduli space of stable curves.
Since then, $\o\M_{g,n}$ itself became the subject of great interest, because of its rich geometry, and because of several applications. In particular, $\o\M_{g,n}$ has a central role in Gromov-Witten theory and enumerative geometry.
In fact, in part motivated by Witten's conjecture (\cite{witten}), the study of the cohomology ring of $\o{\M}_{g,n}$ attracted the attention of several mathematicians in the last decades and led to very important results. 
We recall, for instance, Kontsevich's first proof of the Witten conjecture in (\cite{konts}); the interaction between geometry and Physics leading to the development of quantum cohomoly and Gromov-Witten theory (see e.g. \cite{fulton}); the algebro-geometric inductive calculations on the cohomology ring of $\o\M_{g,n}$ due to Arbarello and Cornalba in \cite{arbcorn}; Faber's conjectures on the structure of the tautological ring of $M_g$ and its pointed versions (\cite{faber}, \cite{panda2}), the ELSV formulas relating intersection formulas in $\o{\mathcal M}_{g,n}$ with Hurwitz numbers (\cite{ELSV1}, \cite{ELSV2}) and the recent proof by Faber, Shadrin and Zvonkine in \cite{genwitten} of the generalized Witten conjecture (\cite{witten2}).

So, it is natural to search for a compactification of $\mathcal Pic_{d,g,n}$ over $\o{\mathcal M}_{g,n}$ and to study its applications. Nevertheless, at least to our knowledge, there was no construction of compactified Picard varieties for curves with marked points until now.

The aim of the present paper is to construct an algebraic (Artin) stack $\o\P_{d,g,n}$ with a geometrically meaningful modular description, giving a solution for the above problem for $g\geq 3$.

Let $n=0$ and consider Pic$_g^d$, the \av universal Picard variety of degree $d$\uv over $ M_g^0$ parametrizing isomorphism classes of line bundles of degree $d$ over automorphism-free nonsingular curves. The problem of compactifying $\Pic_g^d$ over $\o M_g$ was addressed by Caporaso in \cite{cap} and later by Pandharipande in \cite{panda} (the second author's construction holds also for bundles of higher rank). Both compactifications were done by means of GIT-quotients yielding projective varieties endowed with a proper map onto $\o M_g$, extending the natural map $\Pic_g^d\to M_g^0$. Even if the boundary objects used by the two authors have different modular interpretations, it turns out that the resulting varieties coincide. 

We here focus on Caporaso's original compactification, which works for $g\geq 3$; denote it by $\o P_{d,g}$ and by $\phi_d$ its natural map onto $\o M_g$. Given $[X]\in \o M_g$, $\phi_d^{-1}(X)$ is a projective connected scheme with a finite number of components (that cannot exceed a certain numerical invariant of the curve) 
and, if $X$ has trivial automorphism group, $\phi_d^{-1}(X)$ is reduced and its smooth locus is isomorphic to the disjoint union of a finite
number of copies of the jacobian of $X$, $J_X$. Later, on \cite{capneron}, the same author gives a stack-theoretical description of the above quotient for degrees $d$ such that $(d-g+1,2g-2)=1$, getting a modular compactification of $\mathcal Pic_{d,g,0}$ over $\o{\mathcal M}_g$. We considered the same problem in \cite{melo} with no assumption on $d$, getting an algebraic (Artin) stack $\o{\G}_{d,g}$ with a universally closed morphism $\Psi_{d,g}$ onto $\o{\mathcal M}_g$ parametrizing families of \textit{quasistable} curves endowed with \textit{balanced} line bundles of relative degree $d$. Our stack $\o\P_{d,g,n}$ coincides with $\o{\G}_{d,g}$ for $n=0$.

Let $d\in \Z$, $n\geq 0$ and $g\geq 3$. We define $\o\P_{d,g,n}$ to be the stack whose sections over a scheme $S$ are given by families of genus $g$ $n$-pointed quasistable curves over $S$ endowed with a relative degree $d$ balanced line bundle (see Definitions \ref{quasistabledef} and \ref{balanced} below). Morphisms between two such sections are given by cartesian diagrams of the families compatible with the sections plus an isomorphism between the line bundles (see Definition \ref{compstack}). The principal result of this paper consists in proving that $\o\P_{d,g,n}$ is a smooth and irreducible algebraic (Artin) stack of dimension $4g-3+n$, endowed with a universally closed morphism onto $\o{\mathcal M}_{g,n}$.

Our interest in constructing such a space is also due to Goulden, Jackson and Vakil's \av generalized ELSV formula\uv conjecturing a relation between the intersection theory of a $(4g-3+n)$-dimensional space and certain double Hurwitz numbers (see \cite{vakil} and \cite{vakil2}). According to these authors, this space should be a suitable compactification of $\mathcal Pic_{d,g,n}$ over $\o\M_{g,n}$ supporting particular families of classes sa\-tis\-fying certain properties.
Unfortunately, we do not know yet if our space supports such classes, except for what they call $\psi$-classes, which turn out to be the pullback of the $\psi$-classes in $\o\M_{g,n}$. It is certainly interesting to consider this as a future research problem.

Our construction of $\o\P_{d,g,n}$ goes in the following way. For $n=0$, we define $\o\P_{d,g,0}$ to be equal to $\o{\G}_{d,g}$. Then, for $n>0$, along the lines of Knudsen's construction of $\o{\mathcal M}_g$ in \cite{knudsen}, we show that there is an isomorphism between $\o\P_{d,g,n}$ and the universal family over $\o\P_{d,g,n-1}$. This isomorphism is built explicitly and it generalizes Knudsen's notion of \textit{contraction} and \textit{stabilization} of $n$-pointed stable curves, in this more general context of quasistable curves endowed with balanced line bundles.

The stacks $\o\P_{d,g,n}$ can never be Deligne-Mumford since there is an action of $\mathbb G_m$ given by scalar product on the line bundles, leaving the curves and the sections fixed.
Even the rigidification in the sense of \cite{avc} of $\o\P_{d,g,n}$ by this action of $\mathbb G_m$, denoted by $\o{\P}_{d,g,n}\fatslash \mathbb G_m$ is not Deligne-Mumford in general. In fact, already for $n=0$, it was proved by Caporaso in \cite{capneron} that $\o\P_{d,g,0}\fatslash \mathbb G_m$ is Deligne-Mumford if and only if $(d-g+1,2g-2)=1$. This holds in general for any $n\geq 0$.

In order to make the contraction process work, we also need to prove some technical properties for balanced line bundles over quasistable pointed curves. In particular, we get general results concerning global and normal generation for line bundles on (reducible) nodal curves depending only on their multidegree (see section \ref{technical}).

In section \ref{forgetful} we show that $\o\P_{d,g,n}$ is endowed with a (forgetful) morphism $\Psi_{d,g,n}$ onto $\o{\M}_{g,n}$.
We also study the fibres of $\Psi_{d,g,n}$.

Finally, in section \ref{properties}, we study further properties of $\o\P_{d,g,n}$. For example, we show that if $d$ and $d'$ are such that $2g-2$ divides $d-d'$, then $\o\P_{d,g,n}$ is isomorphic to $\o\P_{d',g,n}$. We also study the map form $\o\P_{d,g,n+1}$ to $\o\P_{d,g,n}$ and its sections and we show that these yield Cartier divisors on $\o\P_{d,g,n+1}$ with interesting intersection properties.

We would also like to observe that another possible approach to the construction of $\o{\mathcal P}_{d,g,n}$ would be to use Baldwin and Swinarski's GIT construction of the moduli space of stable maps and, in particular, of $\o M_{g,n}$, in \cite{baldwin}, and then procede as Caporaso in \cite{cap}.

\section{Preliminaries and introduction to the problem}

We will always consider schemes and algebraic stacks locally of finite type over an algebraically closed base field $k$.

A curve $X$ will always be a connected projective curve over $k$ having at most nodes as singularities. 

\subsubsection{Line bundles on reducible curves}

Given a curve $X$, we will denote by $\omega_X$ its canonical or dualizing sheaf.
For each proper subcurve $Z$ of $X$ (which we always assume to be complete and connected), denote by $k_Z:=\sharp (Z\cap \o{X\sm Z})$ and by $g_Z$ its arithmetic genus.
Recall that, since $Z$ is connected, the adjunction formula gives
\begin{equation}\label{omegaZ}
w_Z:=\deg_Z\omega_X=2g_Z-2+k_Z.
\end{equation}

For $L\in \Pic X$ its \textit{multidegree} is $\underline{\deg}L:=(\deg_{Z_1}L,\dots,\deg_{Z_\gamma}L)$ and its (total) degree is $\deg L:=\deg_{Z_1}L+\dots +\deg_{Z_\gamma}L$, where $Z_1,\dots,Z_\gamma$ are the irreducible components of $X$.

Given $\underline{d}=(d_1,\dots, d_\gamma)\in \Z^\gamma$, we set $\Pic^{\underline d}X:=\{L\in\Pic X: \underline{\deg}L=\underline{d}\}$ and $\Pic^dX:=\{L\in\Pic X:\deg L=d\}$. We have that $\Pic^d X=\sum_{|\underline d |=d}\Pic^{\underline d}X$, where $|\underline d|=\sum_{i=1}^\gamma d_i$.

\subsection{Statement of the problem}\label{problem}

Let $\mathcal Pic_{d,g,n}$ be the universal Picard stack over the moduli space of smooth curves with $n$ marked points. Sections of $\mathcal Pic_{d,g,n}$ over a scheme $S$ consist of flat and proper families $\pi:C\rightarrow S$ of smooth curves of genus $g$, with $n$ distinct sections $s_i:S\rightarrow C$ and a line bundle $L$ of relative degree $d$ on $C$. Morphisms between two such objects are given by cartesian diagrams
\begin{equation*}
\xymatrix{
{C} \ar[d]^{\pi} \ar[r]^{\beta_2}  & { C'} \ar[d]_{\pi'} \\
{S} \ar[r]_{\beta_1} \ar @{->} @/^/[u]^{s_i} & {S'} \ar @{->} @/_/[u]_{s_{i'}}
}
\end{equation*}
such that $s_{i'} \circ \beta_1=\beta_2\circ s_i$, $1\le i \le n$, together with an isomorphism $\beta_3:L\rightarrow \beta_2^*(L')$.

In the present paper we search for a compactification of $\mathcal Pic_{d,g,n}$ over $\o{\M}_{g,n}$. By this we mean to construct an algebraic stack $\o{\mathcal P}_{d,g,n}$ with a map $\Psi_{d,g,n}$ onto $\o\M_{g,n}$ with the following properties.
\begin{enumerate}
\item $\o{\mathcal P}_{d,g,n}$ and $\Psi_{d,g,n}$ fit in the following diagram;
\begin{equation}\label{diagram}
\xymatrix{
{\mathcal Pic_{d,g,n}} \ar[d] \ar @{^{(}->}[r]  & {\o{\mathcal P}_{d,g,n}} \ar[d]^{\Psi_{d,g,n}} \\
{\mathcal M_{g,n}}  \ar @{^{(}->}[r] & {\o{\mathcal M}_{g,n}}
}
\end{equation}
\item $\Psi_{d,g,n}$ is proper (or, ate least, universally closed);
\item $\mathcal Pic_{d,g,n}$ has a geometrically meaningful modular description.
\end{enumerate}

Note that in order to complete $\mathcal Pic_{d,g,n}$ over $\o{\mathcal M}_{g,n}$ it is not enough to consider
the stack of line bundles over families of $n$-pointed stable curves, since this is not complete as well.
So, it is necessary to enlarge the category either admitting more general sheaves than line bundles
or a bigger class of curves.

\subsubsection{Strategy of the construction}

Our construction of $\o{\mathcal P}_{d,g,n}$ will go as follows. We will start by noticing that, for $n=0$, this problem is already solved. Then, for $n>0$, we will construct
$\o{\mathcal P}_{d,g,n}$ by induction on the number of marked points $n$, following the lines of Knudsen's construction
of $\o{\M}_{g,n}$.

\subsection{The case $n=0$}\label{n=0}

The problem of constructing compactifications of the Picard varieties of singular curves has been widely studied and, in fact, there are several known compactifications.

We will focus on the one built by Caporaso in \cite{cap}, which we will now briefly describe.

Consider the Hilbert scheme $H$ of genus $g$ curves embedded in $\mathbb P^r$ as nondegenerate curves of degree $d$, where $r=d-g$. There is a natural action of $PGL(r+1)$ in $H$ corresponding to the choice of the coordinates used to embed the curves.
For $d>>0$, define $$\o P_{d,g}:=H_d/PGL(r+1)$$ as the GIT-quotient
of $H_d$, the locus of GIT-semistable points for this action. By results of \cite{gieseker} and \cite{cap}, we know that points in $H_d$ correspond exactly to \textit{quasistable} curves embedded by \textit{balanced} line bundles (of degree $d$), where quasistable curves are semistable curves such that two exceptional components never meet and balanced is a combinatorial condition on the multidegree of the line bundle on the curve (see Definition \ref{balanced} and (\ref{giesekerbasic}) below). 

By construction, $\o P_{d,g}$ has a proper morphism
$\phi_d:\o P_{d,g}\to\o M_g$ such that
$\phi_d^{-1}(M_g^0)$ is isomorphic to the \av universal Picard
variety of degree $d$\uv, $\Pic_g^d$, which parametrizes isomorphism
classes of line bundles of degree $d$ over automorphism-free
nonsingular curves.

Later, in \cite{capneron}, Caporaso shows that, if $d>>0$ and $(d-g+1,2g-2)=1$, the quotient stack associated to the GIT-quotient above, $[H_d/PGL(r+1)]$, is a smooth and irreducible Deligne-Mumford stack endowed with a proper and strongly representable map onto $\o{\M}_g$. Moreover, $[H_d/PGL(r+1)]$ has the following modular description.

Consider the stack $\o{\G}_{d,g}$ (over $SCH_k$) whose sections over a $k-$scheme $S$ consist on families $\pi:X\rightarrow S$ of genus $g$ quasistable curves over $S$ endowed with a balanced line bundle of relative degree $d$ over $X$ and whose morphisms consist on cartesian diagrams of the curves plus an isomorphism between the line bundles (as in $\mathcal Pic_{d,g,n}$ above, ignoring the sections). There is a natural action of $\mathbb G_m$ on $\o\G_{d,g}$ given by fiberwise scalar multiplication on the line bundles. Then $[H_d/PGL(r+1)]$ is the rigidification (in the sense of \cite{avc}) of $\o\G_{d,g}$ by the action of $\mathbb G_m$.

This result uses mainly the fact that the above GIT-quotient is geometric if and only if $(g-g+1,2g-2)=1$.
In \cite{melo} we showed that, for $d>>0$, even if $(d-g+1,2g-2)\neq 1$, $\o\G_{d,g}$ is isomorphic to $[H_d/GL(r+1)]$, where $GL(r+1)$ acts by projection on $PGL(r+1)$. This implies that $\o\G_{d,g}$ is a smooth and irreducible Artin stack endowed with a universally closed map onto $\o\M_g$. Since for $d$ and $d'$ such that $d\pm d'=m(2g-2)$ for some $m\in \Z$, $\o{\G}_{d,g}$ is isomorphic to $\o{\G}_{d',g}$, we get that the same statement holds in general for any $d\in \Z$.

\section{$n$-pointed quasistable curves and balanced line bundles}

In the present section we will generalize the notions of quasistable curve and of balanced line bundle for curves with marked points. 

Our definitions will be done in such a way that, for $n=0$, these coincide with the old notions introduced by Gieseker and Caporaso (see Definition 1.1 in \cite{melo}), and that, for $n>0$, $n$-pointed quasistable curve are the ones we get by applying the stabilization morphism defined by Knudsen in \cite{knudsen} (see \ref{proof} below) to $(n-1)$-pointed quasistable curves endowed with an extra section without stability conditions. Moreover, balanced line bundles on $n$-pointed quasistable curves correspond to balanced line bundles on the quasistable curves obtained by forgetting the points and by contracting the rational components that get \textit{quasidestabilized} without the points (see Lemma \ref{balcorresp}).

We start by recalling the definitions of (semi)stable pointed curve and by introducing some notation.

\begin{defi}\label{pointedcurves}
An \textbf{$n$-pointed curve} is a connected, projective and reduced nodal curve $X$ together with $n$ distinct marked points $P_i\in X$ such that $X$ is smooth at $P_i$, $1\le i\le n$.

Let $g$ and $n$ be such that $2g-2+n>0$. An \textbf{$n$-pointed stable curve} (resp. \textbf{semistable}) is an $n$-pointed curve such that the number of points where a nonsingular rational component $E$ of $X$ meets the rest of $X$ plus the number of points $P_i$ on $E$ is at least $3$ (resp. $2$).

A \textbf{family of $n$-pointed stable} (resp. \textbf{semistable}) curves is a flat and proper morphism $\pi:X\rightarrow S$ together with $n$ distinct sections $s_i:S\rightarrow X$ such that the geometric fibers $X_s$ together with $s_i(s)$, $1\le i\le n$, are $n$-pointed stable (resp. semistable) curves.

\end{defi}

Note that $n$-pointed (semi)stable curves admit chains of smooth rational curves meeting the rest of the curve in one or two points. By \textbf{rational tail} we mean a proper subcurve $T$ of $X$ with $p_a(T)=0$ and $k_T=1$. Instead, proper subcurves $B$ of $X$ with $p_a(B)=0$ and $k_B=2$ will be denoted by \textbf{rational bridges}.

If $X$ is a strictly $n$-pointed semistable curve, a nonsingular rational component $E$ such that the number of points where $E$ meets the rest of $X$ plus the number of marked points $P_i$ on $E$ is exactly $2$ is called a \textbf{destabilizing} component. An \textbf{exceptional component} is a destabilizing component without marked points.

\begin{defi}\label{quasistabledef}
An $n$-pointed \textbf{quasistable} curve is an $n$-pointed semistable curve $X$ such that
\begin{enumerate}
\item all destabilizing components are exceptional;
\item exceptional components can not be contained in rational tails;
\item each rational bridge contains at most one exceptional component.
\end{enumerate}
A \textbf{family of $n$-pointed quasistable curves} is a proper and flat morphism with $n$ distinct sections whose geometric fibers are $n$-pointed quasistable curves.
\end{defi}

See Figure \ref{notquasistable} for examples of pointed semistable curves which are not quasistable.

\begin{figure}
\begin{center}
\scalebox{.6}{\epsfig{file=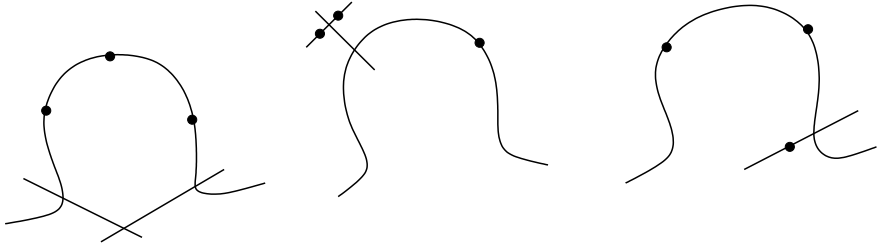}}
\end{center}
\caption{Examples of $3$-pointed semistable curves which are NOT quasistable}
\label{notquasistable}
\end{figure}

To each proper subcurve $Z$ of $X$, denote by $t_Z$ the number of rational tails meeting $Z$.

\begin{defi} \label{balanced}
Let $X$ be an $n$-pointed quasistable curve of genus $g\ge 3$ and $L$ a line bundle on $X$ of degree $d$.
We say that $L$ (or its multidegree) is \textbf{balanced} if the following conditions hold:
\begin{itemize}
\item deg$_EL=1$ for every exceptional component $E$ of $X$;
\item The degree of $L$ on rational bridges can be either $0$ or $1$;
\item If $T$ is a rational tail of $X$, then deg$_TL=-1$;
\item for every connected proper subcurve $Z$ of $X$, if $Z$ is not contained in any rational tail and in any rational bridge of $X$, the following inequality holds
\begin{equation}\label{basic}
|\mbox{deg}_ZL-\frac{d(w_Z-t_Z)}{2g-2}-t_Z|\leq \frac{k_Z-t_Z-2b^L_Z}{2}
\end{equation}
where $b_Z^L$ denotes the number of rational bridges where the degree of $L$ is zero meeting $Z$ in two points.
\end{itemize}
\end{defi}

Note that if $n=0$, $t_Z$ and $b_Z^L$ are equal to $0$ for all proper subcurves $Z$ of $X$, and inequality (\ref{basic}) reduces to the \av Basic Inequality\uv introduced by Gieseker in \cite{gieseker}. In fact, for $n=0$, Definition \ref{balanced} coincides with the definition of balanced multidegree for quasistable curves introduced by Caporaso in \cite{cap} (see also \cite{melo}, Def. 1.1).

Using the notation of \ref{balanced}, denote by

$$m_Z(d,L):=\frac{d w_Z+(3g-3-d)t_Z}{2g-2}+b_Z^L-\frac {k_Z}2$$
and by
$$M_Z(d,L):=\frac{d w_Z+(g-1-d)t_Z}{2g-2}-b_Z^L+\frac{k_Z}2.$$
Then, inequality (\ref{basic}) can be rewritten in the following way
$$m_Z(d,L)\le \mbox{deg}_ZL\le M_Z(d,L)$$

\begin{ex}\upshape{
In figure \ref{balex} we can see an example of a $12$-pointed quasistable curve $X$ consisting of two components of genus bigger than $0$, $C$ and $D$, intersecting each other in $1$ point and other rational components belonging to rational tails or rational bridges. The numbers on the figure indicate the multidegrees of a balanced line bundle on rational tails and on rational bridges. They are uniquely determined with the exception of the rational bridge where there is no exceptional component. In this case, other possibilities would be either $(1,0)$ or $(0,1)$.

\begin{figure}
\begin{center}
\scalebox{.6}{\epsfig{file=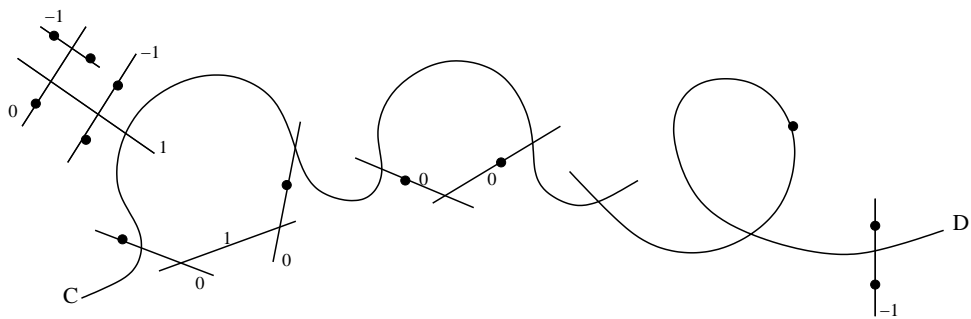}}
\end{center}
\caption{$12$-pointed quasistable curve with assigned balanced multidegree in rational tails and rational bridges.}
\label{balex}
\end{figure}

Consider $d=0$. Then, from Inequality (\ref{basic}), we see that the only possibility for a balanced line bundle of degree $0$ on $X$ completing the multidegree of the figure is to assign to $C$ degree $0$ and to $D$ degree $-1$. In fact, inequality (\ref{basic}) states that the degree of $L$ on $C$ can be either $0$, $1$ or $2$, while on $D$ it must be $0$ or $1$, so $(0,1)$ is the only possible choice in order to the total degree sum up to $0$. If, insted, we had chosen the degree in the rational bridge with no exceptional component to be $1$, then $L$ should have degree $-1$ on $C$. However, in this case inequality (\ref{basic}) would change to $C$: it would give $-1,0,1,2,3$ as possible degrees.

Consider now the case $d=g-1$. Then, since $g=g_C+g_D+2$, we can write $g-1$ as $g_C+g_D+1$. However, since the multidegrees assigned in the figure to rational tails and rational bridges sum up to $-1$, the sum of the degree of $L$ on $C$ with the degree of $L$ on $D$ must be $g_C+g_D+2$. Inequality (\ref{basic}) asserts that the degree of $L$ on $C$ must be in between $g_C+1$ and $g_C+4$ while on $D$ it must be $g_D$ or $g_D+1$. So, we have two possibilities: $(g_C+2,g_D)$ and $(g_D+1,g_D+1)$. If, instead, we had chosen the degree on the rational bridge to be $1$, then the sum of the degree of $L$ on $C$ with the degree of $L$ on $D$ should be $g_C+g_D+1$. However, inequality (\ref{basic}) would change to $C$, giving $g_C,\dots,g_C+5$ as the possible degrees on $C$. So, also in this case we would have two possibilities for the degrees of a balanced line bundle of total degree $g-1$ on $C$ and $D$: $(g_C,g_D+1)$ and $(g_C+1,g_D)$.}
\end{ex}

\subsection{First properties of balanced line bundles on $n$-pointed quasistable curves}

\begin{lem} \label{combdesc}
Let $X$ be an $n$-pointed quasistable curve and suppose $X$ admits a balanced line bundle $L$ on $X$ of degree $d$, for some $d\in \mathbb Z$. Then, if $Z$ is a proper subcurve of $X$ that is contained in a rational tail, then $\deg_ZL=k_Z-2$ and if $Z$ is contained in a rational bridge, then $\deg_ZL$ is either equal to $k_Z-2$ or $k_Z-1$.

In particular, the multidegree of $L$ on rational tails is unique and does not depend on $d$.
\end{lem}

\begin{proof}
Let us begin by showing that the multidegree of $L$ on rational tails is uniquely determined. So, suppose $T$ is a rational tail of $X$. If $T$ is irreducible, then the multidegree of $L$ on $T$ is just the degree of $L$ on $T$, which is necessarily $-1$.

Now, suppose $T$ is reducible. Then there is exactly one irreducible component $E$ of $T$ meeting the rest of the curve (in exactly one point). We will call $E$ the foot of the rational tail. $E$ is a smooth rational curve meeting the rest of $T$ in $k_E-1$ points: denote by $E_1,\dots,E_{k_E-1}$ the irreducible components of $T$ meeting $E$. Then, each $E_i$, $i=1,\dots, k_E-1$ is the foot of a rational tail contained in $T$. In fact, each one of these, if not irreducible, is attached to another rational chain that cannot intersect the rest of the curve since in that case $T$ would contain cycles (which would force $p_a(T)$ to be bigger than $0$). So, $T$ is the union of $E$ with $k_{E}-1$ rational tails meeting $E$, and  $$-1=\deg_TL=\deg_EL+\deg_{\o{T\sm E}}L=\deg{_EL}-(k_{E}-1)$$
which implies that
$$\deg_EL=k_E-2.$$
Note that we don't have to check if inequality (\ref{basic}) is satisfied since it does not apply for subcurves of $X$ contained in rational tails.

Now, iterating the same procedure, it is clear that the degree of each irreducible component of $T$ will be determined since $T$ is the union of $E$ with other $k_E-1$ rational tails with foots $E_1,\dots,E_{k_E-1}$.

Now, consider a rational bridge $B$. Then, $B$ meets the rest of the curve in two points, $p_1$ and $p_2$, and these are linked by a chain of (rational) irreducible components of $B$, $E_1,\dots, E_{l_B}$, each one meeting the previous and the next one, for $i=2,\dots ,l_B-1$. Moreover, each $E_i$ can have rational tails attached. Denote by $B_1,\dots,B_{l_B}$ respectively the proper subcurves of $B$ consisting of $E_i$ and the rational tails attached to it, for $i=1,\dots,l_B$. So, $B=B_1\cup\dots\cup B_{l_B}$ is the union of $l_B$ rational bridges of length $1$.

By definition, the degree of $L$ in $B$ can be either $0$ or $1$, and the same holds for each $B_i$, $i=1,\dots,l_B$. If $\deg_{B_i}L=0$, then, in order to the multidegree of $L$ on $B_i$ sum up to $0$, $\deg_{E_i}L$ must be equal to the number of rational tails attached to it: $t_{E_i}=k_{E_i}-2$. If, instead, the degree of $L$ on $B_i$ is equal to one, then deg$_{B_i}L$ must be equal to $t_{E_i}+1=k_{E_i}-1$. Note that inequality (\ref{basic}) gives that
$$t_{E_i}-1\leq\deg_{E_i}L\leq t_{E_i}+1$$
for $i=1,\dots,l_Z$, so either $t_{E_i}$ or $t_{E_i}+1$ are allowed.
The multidegree of $L$ on the rest of $B_i$ is fixed since $\o{B_i\sm E_i}$ consists of rational tails (that, of course, cannot intersect each other).

Now, if $B$ contains one exceptional component $E$ among the $E_i$'s, say $E_j$, the degree of $B$ must be necessarily $1$ (note that on each rational tail we can have at most one exceptional component by definition of pointed quasistable curve). In this case, we must have that $k_{E_j}=2$, which implies that $E_j$ has no rational tails attached, and the degree of $L$ on it must be $1$. Moreover, the degree of $L$ on the other rational subcurves $B_i$, for $i\neq j$, must be $0$.

If, instead, $B$ does not contain any exceptional component, then we can choose the degree of $L$ in $B$ to be either $1$ or $0$. If we choose it to be $0$, then the degree of $L$ on each $B_i$ must be $0$, for $i=1,\dots,l_B$. If we choose it to be one, we can freely choose one of the $B_i$'s where the degree of $L$ is $1$ and in all the others the degree of $L$ must be $0$.
\end{proof}

\begin{lem}\label{balcorresp}
Let $X$ be an $n$-pointed quasistable curve with assigned multidegree on rational bridges. Let $X'$ be the quasistable curve obtained by contracting all rational tails and rational bridges with assigned degree zero and by forgetting the points. Then, for each degree $d$, the set of balanced multidegrees on $X'$ summing up to $d$ and the set of balanced multidegrees on $X$ summing up to $d$ with the given assigned multidegree on rational bridges are in bijective correspondence.
\end{lem}

\begin{proof}

Let $L'$ be a balanced line bundle on $X'$ with degree $d$. This means that, given a proper subcurve $Z'$ of $X'$, Gieseker's \av Basic Inequality\uv holds for $Z'$, that is,
\begin{equation}\label{giesekerbasic}
-\frac{k_{Z'}}{2}+\frac{dw_{Z'}}{2g-2}\le\deg_{Z'}L'\le \frac{dw_{Z'}}{2g-2}+\frac{k_{Z'}}{2}
\end{equation}
and that the degree of $L'$ on exceptional components is equal to $1$.

Let $C_i$ be an irreducible component of $X=C_1\cup \dots \cup C_\gamma$ such that $C_i$ is not contained in any rational tail and in any rational bridge. Define the multidegree $\underline{d}=(d_1,\dots,d_\gamma)$ on $X$ by declaring that $d_i=\deg_{C_i'}L'+t_{C_i}$ where $C_i'$ is the image of $C_i$ on $X'$. Then we easily see that this defines a balanced multidegree on $X$ (note that the multidegree of $L$ on rational bridges is fixed by hypothesis). In fact, since $k_{C_i}=k_{C_i'}+t_{C_i}+2r^L_{C_i}$ and $g_{C_i}=g_{C_i'}-b_{C_i}^L$, we have that
$$d_i=\mbox{deg}_{C_i' L}+t_{C_i}\le
\frac{dw_{C_i'}}{2g-2}+\frac{k_{C_i'}}{2}+t_{C_i}=$$
$$=\frac{d(2g_{C_i'}-2+k_{C_i'})}{2g-2}+\frac{k_{C_i}-t_{C_i}}{2}-b_{C_i}^L+t_{C_i}=$$
$$=\frac{d(2g_{C_i}+2b_{C_i}^L-2+k_{C_i}-t_{C_i}-2b_{C_i}^L)}{2g-2}+\frac{k_{C_i}}{2}+\frac{t_{C_i}}{2} -b_{C_i}^L=$$
$$=
\frac{dw_{C_i}}{2g-2}+\frac{k_{C_i}}{2}-\frac{d}{2g-2}t_{C_i}+\frac{t_{C_i}}{2}-b_{C_i}^L=$$
$$=\frac{dw_{C_i}}{2g-2}+\frac{k_{C_i}}{2}+\frac{g-1-d}{2g-2}t_{C_i}-b_{C_i}^L$$
and also that
$$d_i\ge
\frac{dw_{C_i'}}{2g-2}-\frac{k_{C_i'}}{2}+t_{C_i}=$$
$$=\frac{dw_{C_i}}{2g-2}-\frac{k_{C_i}}{2}+\frac{3g-3-d}{2g-2}t_{C_i}+b_{C_i}^L$$
so inequality (\ref{basic}) holds for $C_i$ if and only if Gieseker's Basic Inequality holds for $C_i'$. It is easy to see that this is true more generally for any proper subcurve $Z$ of $X$ not contained in any rational tail and in any rational bridge.
\end{proof}

\section{Balanced Picard stacks for $n$-pointed quasistable curves}

\begin{defi}\label{compstack} For any integer $d$ and $g\geq 3$,
denote by $\o\P_{d,g,n}$ the following category fibered in groupoids over the category of schemes over $k$.
Objects over a $k$-scheme $S$ are families $(\pi:X\rightarrow S, s_i:S\rightarrow X,L)$ of $n$-pointed quasistable curves over $S$ and a balanced line bundle $L$ on $X$ of relative degree $d$.

Morphisms between two such objects are given by cartesian diagrams
\begin{equation*}
\xymatrix{
{ X} \ar[d]^{\pi} \ar[r]^{\beta_2}  & { X'} \ar[d]_{\pi'} \\
{S} \ar[r]_{\beta_1} \ar @{->} @/^/[u]^{s_i} & {S'} \ar @{->} @/_/[u]_{s_{i'}}
}
\end{equation*}
such that $s_{i'} \circ \beta_1=\beta_2\circ s_i$, $1\le i \le n$, together with an isomorphism $\beta_3:L\rightarrow \beta_2^*(L')$.
\end{defi}

Note that $\o\P_{d,g,n}$ contains $\mathcal Pic_{d,g,n}$ for all $n\geq 0$.

In what follows we will try to prove the following statement.

\begin{teo}\label{aim}
$\o\P_{d,g,n}$ is a smooth and irreducible algebraic (Artin) stack of dimension $4g-3+n$ endowed with a universally closed map onto $\o{\mathcal M}_{g,n}$.
\end{teo}

Recall that,
for $n=0$, $\o\P_{d,g,n}$ coincides with the stack $\o{\G}_{d,g}$ defined in \cite{melo}, section 3 (see \ref{n=0} above), so Theorem \ref{aim} holds in this case.

Now, following Knudsen's construction of $\o{\M}_{g,n}$ (see \cite{knudsen}), we will show that Theorem \ref{aim} holds for all $n\geq 0$ using the following induction argument. We will prove that, for $n>0$, $\o\P_{d,g,n+1}$ is isomorphic to the universal family over $\o\P_{d,g,n}$.

By universal family over $\o\P_{d,g,n}$ we mean an algebraic stack $\mathcal Z_{d,g,n}$ with a map onto $\o\P_{d,g,n}$ admitting $n$-sections $\sigma_{d,g,n}^i:\o\P_{d,g,n}\rightarrow \mathcal Z_{d,g,n}, i=1,\dots ,n$ and endowed with an (universal) invertible sheaf $\mathcal L$ such that, given a family $f:C\rightarrow S, s_i:S\rightarrow C,i=1,\dots,n$ of $n$-pointed quasistable curves and a balanced line bundle $L$ over $C$ of relative degree $d$,
the following diagram, commuting both in the upward and downward directions,
\begin{equation}\label{universal}
\xymatrix{
{C} \ar[r]^{\pi_2} \ar[d]^{f} & {\mathcal{Z}_{d,g,n}} \ar[d]\\
{S} \ar[r]_{\mu_f} \ar @{->} @/^/[u]^{s_i} & {\o\P_{d,g,n}} \ar @{->} @/_/[u]_{\sigma_{d,g,n}^i}
}
\end{equation}
is cartesian and induces an isomorphism between $\pi_2^*(\mathcal L)$ and $L$.

Let $\mathcal Z_{d,g,n}$ be the category whose sections over a scheme $Y$ are families of $n$-pointed quasistable curves $X\rightarrow Y,t_i:Y\rightarrow X, i=1,\dots ,n$ endowed with a balanced line bundle $M$ of relative degree $d$ and with an extra section $\Delta:Y\rightarrow X$. Morphisms in $\mathcal Z_{d,g,n}$ are like morphisms in $\o\P_{d,g,n}$ compatible with the extra section. $\mathcal Z_{d,g,n}$ is an algebraic stack endowed with a forgetful morphism onto $\o\P_{d,g,n}$ admitting $n$ sections given by the diagonals $\delta_{1,n+1}, \dots, \delta_{n,n+1}$.

It is easy to see that, given a family of $n$-pointed quasistable curves $f:C\rightarrow S, s_i:S\rightarrow C,i=1,\dots,n$ and a balanced line bundle $L$ over $C$ of relative degree $d$, diagram (\ref{universal}) is cartesian, where $\pi_2$ is defined by associating to the identity morphism $1_C:C\rightarrow C$ the fiber product of $f:C\rightarrow S$ with itself

\begin{equation}\label{identity}
\xymatrix{
{C\times_S C} \ar[r]^{p_2} \ar[d]^{p_1} & {C} \ar[d]_{f}\\
{C} \ar[r]_{f} \ar @{->} @/^/[u]^{f^*s_i} & {S} \ar @{->} @/_/[u]_{s_i}
}
\end{equation}
endowed with an extra section $\Delta:C\rightarrow C\times_SC$ given by the diagonal and with the relative degree $d$ line bundle $p_2^*(L)$. Given another object $h:Y\rightarrow C$ of $C$, $\pi_2(h)$ is defined to be the fiber product of $h$ and $p_1$ defined in (\ref{identity}), naturally endowed with the $n+1$ pullback sections and with the pullback of $p_2^*(L)$.

The universal sheaf over $\mathcal Z_{d,g,n}$, $\mathcal L$, is defined by associating to each section $(X\rightarrow Y, t_i,M,\Delta)$ of $\mathcal Z_{d,g,n}$ over $Y$,
the line bundle $\Delta^*(M)$ over $Y$. It is easy to see that this defines an invertible sheaf on $\mathcal Z_{d,g,n}$.

Now we easily check that $\mathcal L$ is the universal sheaf over $\mathcal Z_{d,g,n}$. Indeed, given an object $h:Y\rightarrow C$ on $C$, $\pi_2^*(\mathcal L)(h)=\mathcal L(\pi_2(h))\cong h^*(L)$, so it is isomorphic to the sheaf defined by $L$ on $C$, considered as a stack.

We have just proved the following.

\begin{prop}
The algebraic stack $\mathcal Z_{d,g,n}$ defined above endowed with the invertible sheaf $\mathcal L$ is the universal family over $\o\P_{d,g,n}$ for the moduli problem of $n$-pointed quasistable curves with a balanced degree $d$ line bundle.
\end{prop}

Now, suppose we can show that, for all $n\geq 0$, there is forgetful morphism $\Psi_{d,g,n}$ from $\o\P_{d,g,n}$ onto $\o{\mathcal M}_{g,n}$
such that the image under $\Psi_{d,g,n}$ of an $n$-pointed quasistable curve $X$ over $S$ endowed with a balanced degree $d$ line bundle is the stable model of $X$ over $S$ forgeting the line bundle. These morphisms would yield commutative diagrams as follows, for all $n>0$.

\begin{equation}\label{induction}
\xymatrix{
& {\o\P_{d,g,n}} \ar[rd]^{\Psi_{d,g,n}} \ar[dl]_{\Phi_{d,g,n}} & \\
{\o\P_{d,g,n-1}} \ar[dr]_{\Psi_{d,g,n-1}} & & {\o{\mathcal M}_{g,n}} \ar[dl]^{\Pi_{g,n}}\\
& {\o{\mathcal M}_{g,n-1}} &
}
\end{equation}

Since $\Pi_{g,n}$ and $\Phi_{d,g,n}$ are the morphisms from the universal families over $\o\P_{d,g,n-1}$ and $\o{\mathcal M}_{g,n-1}$, respectively, it follows that $\Psi_{d,g,n}$ is universally closed (or proper) if and only if $\Psi_{d,g,n-1}$ is. Since, by \cite{melo}, $\Psi_{d,g,0}$ is universally closed, we have that $\Psi_{d,g,n}$ is universally closed for all $n\geq 0$.

So, our proof of Theorem \ref{aim} will follow if we can prove the following statement, which is the aim of the present paper.

\begin{teo}\label{main}
For all $n>0$, $\o\P_{d,g,n+1}$ is isomorphic to the algebraic stack $\mathcal Z_{d,g,n}$.
\end{teo}

\begin{rem}\upshape{
Recall that, if $n=0$, in \cite{melo}, $\o\P_{d,g,0}=\o{\G}_{d,g}$ is shown to be isomorphic to the quotient stack $[H_d/GL(r+1)]$, where $H_d$ is a certain subscheme of $\Hilb_r^{dt-g+1}$. Note that the action of $GL(r+1)$ in $H_d$ naturally lifts to an action in $\mathcal Z_d$, where $\mathcal Z_d$ is the restriction to $H_d$ of the universal family over the Hilbert scheme. Using a similar proof we can show that $\mathcal Z_{d,g,1}$ is isomorphic to the quotient stack $[\mathcal Z_d/GL(r+1)]$.}
\end{rem}

\section{Properties of line bundles on reducible nodal curves}\label{technical}

In this section we prove some technical properties of line bundles over (reducible) nodal curves that will be used later in the proof of Theorem \ref{main}.

\begin{lem}\label{basicpropgen}
Let $X$ be a nodal curve of genus $g$ and $L\in \Pic^dX$. If $\deg_ZL\geq 2g_Z-1$ for every connected subcurve $Z\subseteq X$, then $H^1(X,L)=0$. Moreover, if strict inequality holds above for all $Z\subseteq X$, $L$ has no base points.
\end{lem}

To prove Lemma \ref{basicpropgen} we will use the following Lemma, which is Lemma 2.2.2 in \cite{cliff}.

\begin{lem}[Caporaso,\cite{cliff}]\label{222}
Let $X$ be a nodal curve of genus $g$ and $L\in \Pic^dX$. If, for every connected subcurve $Z$ of $X$, $\deg_ZL\geq 2g-2$, then $h^0(X,L)=d-g+1$.
\end{lem}

\begin{proof}[Proof (of Lemma \ref{basicpropgen})]
The first assertion follows immediatly  by Serre duality and by Lemma \ref{222}.

Now, assume that, for every $Z\subseteq X$, $\deg_ZL\geq 2g$. We must show that $L$ has no base points. 
Consider a closed $k$-rational point $x$ in $X$.
Suppose that $x$ is a nonsingular point of $X$. We must show that
$$h^0(X,L(-x))<h^0(X,L).$$
By our assumption on $L$, we can apply again Lemma \ref{222} to $L(-x)$ to get that $h^0(X,L(-x))=d-1-g+1=h^0(X,L)-1$.

Now, suppose $x$ is a nodal point of $X$. We must show that
$$H^0(X,L\otimes \mathcal I_x)\subsetneq H^0(X,L).$$
By contradiction, suppose these are equal. Then, if $\nu:Y\rightarrow X$ is the partial normalization of $X$ at $x$, we get that
$$H^0(X,L)= H^0(Y,\nu^*L(-p-q)),$$
where $p$ and $q$ are the preimages of $x$ by $\nu$.

Suppose that $x$ is not a disconnecting node for $X$. Then, it is easy to see that we can apply Lemma \ref{222} to $(Y, \nu^*L(-p-q))$. In fact, since $x$ is not a disconnecting node for $X$, if $Z'\subseteq Y$ contains $p$ and $q$, then $\deg_{Z'}\nu^*L(-p-q)\geq 2g_{Z'}-1$ because $g_{Z'}=g_Z-1$, where $Z\subseteq X$ is a subcurve of $X$ such that $Z'=\nu^{-1}(Z)$. So, we get that $h^0(Y,\nu^*L(-p-q))=(d-2)-(g-1)+1=d-g$, leading us to a contradiction.

Suppose now that $x$ is a disconnecting node for $X$. Then, $Y$ is the union of two connected curves $Y_1$ and $Y_2$ of genus $g_1$ and $g_2$, respectively, with $g_1+g_2=g$. Suppose that $p\in Y_1$ and $q\in Y_2$. Then, 
$$h^0(Y, \nu^*L(-p-q))=h^0(Y_1,\nu^*( L)_{|Y_1}(-p))+h^0(Y_2,\nu^*( L)_{|Y_2}(-q)).$$
Also in this case, we can apply Lemma \ref{222} to $(Y_i,h^0(Y_1,\nu^*( L)_{|Y_1}(-p)))$ and to $(Y_2,h^0(Y_1,\nu^*( L)_{|Y_2}(-q)))$. We get that 
$$h^0(Y, \nu^*L(-p-q))=(\deg_{Y_1}(\nu^*L)-1-g_1+1)+(\deg_{Y_2}(\nu^*L)-1-g_2+1)=d-g,$$
a contradiction.
\end{proof}

\begin{cor}\label{basicdualizing}
Let $X$ be an $n$-pointed quasistable curve of genus $g$ with $2g-2+n>0$ and let $M$ be either $\omega_X(p_1+\dots +p_n)$ or $\omega_X(p_1+\dots + p_{n-1})$ (if $n>0$), where $p_1,\dots, p_n$ are the $n$ marked points of $X$. Then, for all $m\geq 2$, we have that
\begin{enumerate}
 \item $H^1(X,M^m)=0$;
 \item $M^m$ is globally generated.
\end{enumerate}
\end{cor}

\begin{proof}
According to Lemma \ref{basicpropgen}, it is enough to show that, given a subcurve $Z$ of $X$, $\deg_ZM^m\geq 2g_Z$, for all $m\geq 2$. It is sufficient to prove the result for $m=2$.

Let $Z$ be a subcurve of $X$. Then,
$$\deg_Z\omega_X=2g_Z-2+k_Z.$$
Then,
$$\deg_Z(M^2)\geq 4g_Z-4+2k_Z=(2g_Z)+(2g_Z-4+2k_Z).$$ 
Now, if $Z$ is not a rational tail of $X$, $2g_Z-4+2k_Z\geq 0$. Instead, if $Z=T$ is a rational tail of $X$, since $X$ is quasistable, it must contain at least two of the $n$ marked points and, if $n>0$, at least one of these is among $p_1,\dots,p_{n-1}$. So, $\deg_TM^2\geq 2(g_Z-2+k_Z+1)=0=2g_Z$ and the result follows.

\end{proof}

\begin{cor} \label{basicprop}
Let $d>>0$, $n>0$ and $X$ an $n$-pointed quasistable curve of genus $g\geq 2$ endowed with a balanced line bundle $L$ of degree $d$. Denote by $p_1,\dots, p_{n}$ the $n$ marked points of $X$ and let $M$ be the line bundle $L(p_1+\dots +p_{n-1})\otimes (\omega_X(p_1+\dots+p_{n-1}))^{-k}$, for any $k\leq 1$. Then, we have that, for all $m\ge 1$,
\begin{enumerate}
\item $H^1(X,M^{m})=0$;
\item $M^{m}$ is globally generated.
\end{enumerate}
\end{cor}

\begin{proof}

Again, accordingly to Lemma \ref{basicpropgen}, the result follows if we prove that, for every subcurve $Z$ of $X$, $\deg_ZM^m\geq 2g_Z$. It is enough to prove the result for $m=1$.

Let $Z$ be a subcurve of $X$ which is not contained in any rational tail or in any rational bridge of $X$. By definition of balanced (see \ref{balanced} above), we have that
$$\deg_ZL= \frac{d}{2g-2}(w_Z-t_Z)+i_{X,Z}=\frac{d}{2g-2}(2g_Z-2+k_Z-t_Z)+i_{X,Z},$$
where $i_{X,Z}$ is independent of $d$.

Take $d>>0$.
Now, if $Z$ is rational, $k_Z-t_Z\geq 3$, so $\deg_ZL>>0$. In fact, if $k_Z=t_Z$, $X$ would be rational, which is impossible; by the other hand, if $k_Z-t_Z$ is $1$ or $2$, $Z$ should be contained in a rational tail or in a rational bridge of $X$, respectively, which cannot be the case by our assumption on $Z$.

Suppose now that $g_Z=1$. Then, $k_Z-t_Z\geq 1$, since otherwise $X$ would have genus $1$, which is impossible as well. So, also in this case, $\deg_ZL>>0$.

Finally, if $g_Z\geq 2$, it follows immediately that $\deg_ZL>>0$ and the same holds for $\deg_ZM$, which is asymptotically equal to $\deg_ZL$ since $d>>0$. 

Suppose now that $Z$ is contained in a rational tail or in a rational bridge of $X$. Then, from Lemma \ref{combdesc}, we have that $\deg_ZM\geq\deg_ZL\geq k_Z-2\geq 0=2g_Z$. Finally, if $k_Z=1$, $\deg_ZL=-1=w_Z$ and $Z$ has at least $2$ marked points of $X$. So, $\deg_ZL-\deg_Z\omega_Z=0$ and $\deg_ZL(p_1+\dots+p_n)>0=2g_Z$.
\end{proof}

\subsection{Normal generation}

Recall the following definition.

\begin{defi}\label{normgen}
A coherent sheaf $\mathcal F$ on a scheme $X$ is said to be \textbf{normally generated} if, for all $m\ge 1$, the canonical map
$$H^0(X,\mathcal F)^{m}\rightarrow H^0(X,\mathcal F^{m})$$
is surjective.
\end{defi}

Note that if we take $\mathcal F$ to be an ample line bundle $L$ on $X$, then if $L$ is normally generated it is, indeed, very ample (see \cite{mumfordlemma}, section 1). In this case, saying that $L$ is normally generated is equivalent to say that the embedding of $X$ via $L$ on $\mathbb P^N$, for $N=h^0(X,L)-1$, is projectively normal.

Normal generation of line bundles on curves has been widely studied. For instance, we have the following theorem of Mumford:

\begin{teo}[Mumford, \cite{mumfordlemma}, Theorem 6] \label{mumfordteo}
Let $X$ be a nonsingular irreducible curve of genus $g$. Then, any line bundle of degree $d\geq 2g+1$ is normally generated.
\end{teo}

Mumford's proof of Theorem \ref{mumfordteo} is based on the following Lemma.

\begin{lem}[Generalized Lemma of Castelnuovo, \cite{mumfordlemma}, Theorem 2]\label{castelnuovo}
Let $N$ be a globally generated invertible sheaf on a complete scheme $X$ of finite type over $k$ and $F$ a coherent sheaf on $X$ such that
$$H^i(X,F\otimes N^{-i})=0 \mbox{ for } i\geq 1.$$
Then,
\begin{enumerate}
 \item $H^i(X,F\otimes N^j)=0$ for $i+j\geq 0$, $i\geq 1$.

 \item the natural map
$$H^0(X,F\otimes N^i)\otimes H^0(X,N)\rightarrow H^0(X,F\otimes N^{i+1})$$
is surjective for $i\geq 0$.
\end{enumerate}
\end{lem}

The proof of the following statement uses mainly the arguments of Knudsen's proof of Theorem 1.8 in \cite{knudsen}.

\begin{prop}\label{normgengen}
Let $X$ be an $n$-pointed semistable curve of genus $g\geq 2$ and $L\in\Pic^d X$. If, for every subcurve $Z$ of $X$, $\deg_ZL\geq 2g$, then $L\otimes \omega_X(p_{1}+\dots + p_{n})$ is normally generated, where $p_{1}\dots,p_{n}$ are the marked points of $X$.
\end{prop}

\begin{proof}
Let $D$ denote the divisor $p_1+\dots+p_n$. Let $Z$ be a subcurve of $X$.
Since the multidegree of $\omega(D)$ is non-negative, $\deg_ZL\otimes\omega(D)\geq\deg_ZL\geq 2g_Z$, so both statements of Lemma \ref{basicpropgen} hold also for $L\otimes \omega(D)$. So, we can apply the generalized Lemma of Castelnuovo with $F=(L\otimes\omega_X(D))^m $ and $N=L\otimes \omega_X(D)$, for any $m>1$, and get that the natural map
$$H^0(X,(L\otimes\omega_X(D))^m)\otimes H^0(X,L\otimes\omega_X(D))\rightarrow H^0(X,(L\otimes\omega_X(D))^{m+1})$$
is surjective. So, to prove that $L\otimes \omega_X(D)$ is normally generated, it remains to show that the map
\begin{equation*}
H^0(X,L\otimes\omega_X(D))\otimes H^0(X,L\otimes\omega_X(D))\stackrel{\alpha}{\rightarrow} H^0(X,(L\otimes\omega_X(D))^2)
\end{equation*}
is surjective.

Start by assuming that $X$ has no disconnecting nodes and consider the following commutative diagram
\begin{equation*}
\xymatrix{
{\Gamma(L\otimes\omega_X(D))\otimes \Gamma(\omega_X)\otimes \Gamma(L(D)) } \ar[d]_{\beta} \ar[r] & { \Gamma(L\otimes\omega_X(D))\otimes \Gamma(L\otimes\omega_X(D))} \ar[d]^{\alpha}\\
{\Gamma(L\otimes\omega_X^2(D)) \otimes \Gamma(L(D))} \ar[r]^{\gamma}  & {\Gamma((L\otimes\omega_X(D))^2)}
}
\end{equation*}
where $\Gamma(-)$ indicates $H^0(X,-)$.

Then, from the proof of Theorem 1.8 in Knudsen we have that, since $g\geq 2$ and $X$ has no disconnecting nodes, $H^0(X,\omega_X)$ is globally generated.

Moreover, from Lemma \ref{basicpropgen} applied to $L(D)$, we get that $H^1((L\otimes\omega_X(D)) \otimes\omega_X^{-1})=H^1(L(D))=0$. So, we can apply the Generalized Lemma of Castelnuovo with $F=L\otimes\omega_X(D)$ and $N=\omega_X$ to conclude that $\beta$ is surjective.

Now, since $X$ has no disconnecting nodes, it cannot have rational tails. So, we can see $X$ as a semistable curve without marked points and apply Corollary \ref{basicdualizing} to $X$ and $\omega_X$ and get that $H^1((L\otimes\omega_X^2(D))\otimes (L(D))^{-1})=H^1(\omega_X^2)=0$. Since $L(D)$ is globally generated, again by Lemma \ref{basicpropgen}, we can apply the Generalized Lemma of Castelnuovo with $F=L\otimes \omega_X^2(D)$ and $N=L(D)$ to conclude that also $\gamma$ is surjective.

Since the above diagram is commutative, it follows that also $\alpha$ is surjective and we conclude.

Now, to show that $\alpha$ is surjective in general, let us argue by induction in the number of disconnecting nodes of $X$.

Let $x$ be a disconnecting node of $X$ and $X_1$ and $X_2$ the subcurves of $X$ such that $\{x\}=X_1\cap X_2$.

The surjectivity of $\alpha$ follows if we can prove the folowing two statements.
\begin{enumerate}
 \item The image of $\alpha$ contains sections a section
$s\in H^0(X,(L\otimes\omega_X(D)^2))$ such that $s(x)\neq 0$;
 \item The image of $\alpha$ contains $H^0(X,(L\otimes \omega_X(D))^2\otimes \mathcal I_{x}))$.
\end{enumerate}

The first statement follows immediately from the fact that $L\otimes \omega_X(D)$ is globally generated (once more by \ref{basicpropgen}).

Let $M$ denote $L\otimes \omega_X(D)$.
To prove $(2)$ let us consider $\sigma\in H^0(X,M^2\otimes \mathcal I_{x})$. Then, $\sigma=\sigma_1+\sigma_2$, with
$$\sigma_1\in H^0(X,M^2\otimes\mathcal I_{X_1})\cong H^0(X_2,(M^2\otimes\mathcal I_{X_1})_{|X_2}),$$
$$\sigma_2\in H^0(X,M^2\otimes\mathcal I_{X_2})\cong H^0(X_1,(M^2\otimes\mathcal I_{X_2})_{|X_1}).$$
By induction hypothesis, $\sigma_1$ is in the image of
$$H^0(X_2,M_{|X_2})\otimes H^0(X_2,M_{|X_2})\to H^0(X_2,(M^2)_{|X_2})$$
and $\sigma_2$ in the image of
$$H^0(X_1,M_{|X_1})\otimes H^0(X_1,M_{|X_1})\to H^0(X_1,(M^2)_{|X_1})$$
with both $\sigma_1$ and $\sigma_2$ vanishing on $x$.

Write $\sigma_1$ as $\sum_{l=1}^r u_l\otimes v_l$, with $u_l$ and $v_l$ in
$H^0(X_2,M_{|X_2})$, for $l=1,\dots ,r$. Let $\nu:Y\to X$ be the partial normalization of $X$ in $x$ and $p$ and $q$ be the preimages of $x$ on $X_1$ and $X_2$, respectively, via $\nu$. Since $M_{|X_1}$ is globally generated, there is $s\in H^0(X_1,M_{|X_1})$ with $s(p)\neq 0$. Then there are constants $a_l$ and $b_l$ for $l=1,\dots,r$ and $i=1,\dots,k$ such that
\begin{equation}\label{incol}
a_l s(p)=u_l(q) \mbox{ and }b_l s(p)=v_l(q).
\end{equation}
Define the sections $\bar u_l$ and $\bar v_l$ as $u_l$ (resp. $v_l$) on $X_2$ and as $a_ls$ (resp. $b_ls$) on $X_1$, for $l=1,\dots,r$. By (\ref{incol}), these are global sections of $M$ and
$$\sum_{l=1}^r\bar u_l\otimes \bar v_l$$ maps to $\sigma_1$. In fact,
$$\sigma_1(x)=\sum_{l=1}^ru_l(q)\otimes v_l(q)==\sum_{l=1}^r(a_ls(p)\otimes b_ls(p))=(\sum_{l=1}^r a_lb_l)s(p)\otimes s(p)$$
and, by hypothesis, $\sigma_1(x)=0$ and $s(p)\neq 0$. This implies that $\sum_l^r a_lb_l=0$, so $(\sum_{l=1}^r\bar u_l\otimes \bar v_l)_{|X_1}=0$. We conclude that $\sigma_1$ is in the image of $\alpha$.

In the same way, we would get that also $\sigma_2$ is in the image of $\alpha$, so $(2)$ holds and we are done.
\end{proof}

The next result follows from the proof of Theorem 1.8 in \cite{knudsen}, however we include it here since we shall use it in the following slightly more general form.

\begin{cor}\label{normgendualizing}
 Let $X$ be an $n$-pointed semistable curve of genus $g\geq 2$ and let $p_1\dots,p_n$ be the marked points of $X$. Then, for $m\geq 3$, $(\omega_X(p_1+\dots+p_n))^m$ is normally generated.
\end{cor}

\begin{proof}
Is an immediate consequence of Proposition \ref{normgengen} and Corollary \ref{basicdualizing}.
\end{proof}

\begin{cor}\label{normgenprop}
Let $X$ be an $n$-pointed quasistable curve of genus $g$ and $L$ a balanced line bundle on $X$ of degree $d>>0$. Then, if $n>0$, $L(p_1+\dots +p_{n-1})$ is normally generated.
\end{cor}

\begin{proof}

Since $X$ is an $n$-pointed quasistable curve, it is easy to see that $X$, endowed with the first $n-1$ marked points $p_1,\dots,p_{n-1})$, is an $(n-1)$-pointed semistable curve. Moreover, by Corollary \ref{basicprop}, we can apply Proposition \ref{normgengen} to $L(p_1+\dots+p_{n-1})\otimes\omega_X^{-1}(p_1+\dots+p_{n-1})$. The result follows immediately now.

\end{proof}

\begin{cor}\label{projnormal}
Let $d>>0$, $n>0$ and $X$ an $n$-pointed quasistable curve of genus $g\geq 2$ endowed with a balanced line bundle $L$. Let $M$ denote the line bundle $L(p_1+\dots+p_{n})$, where $p_1,\dots,p_n$ are the marked points of $X$. We have:
\begin{enumerate}
 \item $M$ is normally generated;
 \item $M$ is very ample.
\end{enumerate}

\end{cor}

\begin{proof}
Statement $(1)$ follows from the proof of the previous proposition, which obviously works for $M=L(p_1+\dots+p_n)$ as well.

To show $(2)$ it is enough to observe that $M$ is ample since its degree on each irreducible component of $X$ is positive. Since $M$ is also normally generated, it follows that $M$ is indeed very ample (see \cite{mumfordlemma}, section 1).
\end{proof}

\section{The contraction functor}

The following definition generalizes the notion of contraction introduced by Knudsen in \cite{knudsen} to the more general case of pointed quasistable curves endowed with balanced line bundles.

\begin{defi}\label{contraction}
Let $(\pi:X\rightarrow S,s_i:S\rightarrow X,L)$ be an $(n+1)$-pointed curve endowed with a line bundle of relative degree $d$. A contraction of $X$ is an $S$-morphism from $X$ into an $n$-pointed  curve $(\pi':X'\rightarrow S,t_i:S\rightarrow X',L')$ endowed with a line bundle of relative degree $d$, $L'$, and with an extra section $\Delta:S\rightarrow X'$ such that
\begin{enumerate}
 \item  for $i=\,\dots,n$, the diagram
\begin{equation*}
\xymatrix{
{X} \ar[r]^{f} \ar[d]^{\pi} & {X'} \ar[dl]_{\pi'}\\
{S} \ar @{->} @/^/[u]^{s_i} \ar @{->} @/_/[ur]_{t_i}
}
\end{equation*}
commutes both in the upward and downward directions,
\item $\Delta=fs_{n+1}$,
\item $f$ induces an isomorphism between $L(s_1+\dots +s_n)$ and $f^*L'(t_1+\dots + t_n)$,
\item the morphism induced by $f$ in the geometric fibers $X_s$ is either an isomorphism or there is an irreducible rational component $E\subset X_s$ such that $s_{n+1}(s)\in E$ which is contracted by $f$ into a closed point $x\in X_s'$ and
$$f_s:X_s\sm E\rightarrow X'_s\sm \left\lbrace  x\right\rbrace$$
is an isomorphism.
\end{enumerate}
\end{defi}

\subsection{Properties of contractions}

\begin{prop}\label{quasistablecont}
Let $S=\spec k$ and $f:X\to X'$ a contraction from an $(n+1)$-pointed curve $(X,p_i,L)$, endowed with a degree $d$ line bundle $L$ into an $n$-pointed curve $(X',q_i,L')$, endowed with a degree $d$ line bundle $L'$ and with an extra point $r$. Then, if $(X,p_i)$ is quasistable, also $(X',q_i)$ is quasistable and, in this case, $L$ is balanced if and only if $L'$ is balanced.
\end{prop}

\begin{proof}
Clearly, the assertion follows trivially if no irreducible component of $X$ gets contracted by $f$. So, assume that there is an irreducible component $E$ of $X$ that gets contracted by $f$.
Then, necessarily, $p_{n+1}\in E$, so no exceptional component of $X$ gets contracted. Moreover, the condition that $f^*L'(q_1+\dots+q_n) \cong L(p_1+\dots + p_n)$ implies that $L(p_1+\dots+p_n)$ is trivial on the fibers of $f$, so it must have degree $0$ on $E$.
Now, we have only two possibilities: either $r$ is a smooth point of $X'$ or it is nodal.

Start by considering the case when $r$ is smooth. Since $f(p_{n+1})=r=f(E)$, we must have that $k_E=1$.
So, if $X$ is quasistable, $E$ must contain exactly another special point $p_i$, for some $i=1,\dots,n$ and $r$ must be a smooth point of $X'$. Let $F'$ be the irreducible component of $X'$ containing $r$ and $F$ the correspondent irreducible component of $X$ (recall that $f$ establishes an isomorphism between $F$ and $F'$ away from $r$). If $g_F>0$, then it is clear that also $X'$ is quasistable. Instead, if $F$ is rational, even if $k_{F'}=k_F-1$, $F'$ has one more marked point than $F$. So, $X'$ has the same destabilizing and exceptional components than $X$ and it follows that $X'$ is a quasistable $n$-pointed curve.

Let us now check that, if we are contracting a rational tail of a quasistable curve, $L$ is balanced if and only if $L'$ is balanced. From the definition of contraction, we get that the multidegree of $L(p_1+\dots +p_n)$ in the irreducible components of $X$ that are not contracted must agree with the multidegree of $L'(q_1+\dots+q_n)$ in their images by $f$. In our case, this implies that the multidegree of $L'$ on the irreducible components of $X'$ coincides with the multidegree of $L$ on the corresponding irreducible components of $X$, except on $F'$, where we must have that
$$\deg_{F'}L'=\deg_FL-1.$$
So, given a proper subcurve $Z'$ of $X'$, if $Z'$ does not contain $r$, the balanced condition will be clearly satisfied by $L$ on $Z$ if and only if it is satisfied by $L'$ on $Z'$ since $m_{Z'}(d,L')=m_Z(d,L)$, $M_{Z'}(d,L')=M_Z(d,L)$ and $\deg_{Z'}(d,L')=\deg_Z(d,L)$. Now, suppose $r\in Z'$ and let $Z$ be the preimage of $Z'$ by $f$. Then, $k_{Z'}=k_Z-1$, $w_Z'=w_Z-1$, $b_Z^L=b_{Z'}^{L'}$ and $t_Z'=t_Z-1$, which implies that
$$m_{Z'}(d,L)=m_{Z}(d,L)-1$$
and
$$m_{Z'}(d,L')=m_Z(d,L)-1.$$
Since also $\deg_{F'}L'=\deg_FL-1$, we conclude that, if $L$ is balanced, then $L'$ is balanced too. Now, to conclude that the fact that $L'$ is balanced implies that also $L$ is balanced we have to further observe that the degree of $L$ on $E$ is forced to be $-1$ since $E$ contains $2$ special points and that the inequality (\ref{basic}) is verified on $\o{X\sm E}$ (that does not correspond to any proper subcurve of $X'$), which follows since $m_{\o{X\sm E}}(d,L)=M_{\o{X\sm E}}(d,L)=d+1=\deg_{\o{X\sm E}}L$.

\begin{figure}
\begin{center}
\scalebox{.6}{\epsfig{file=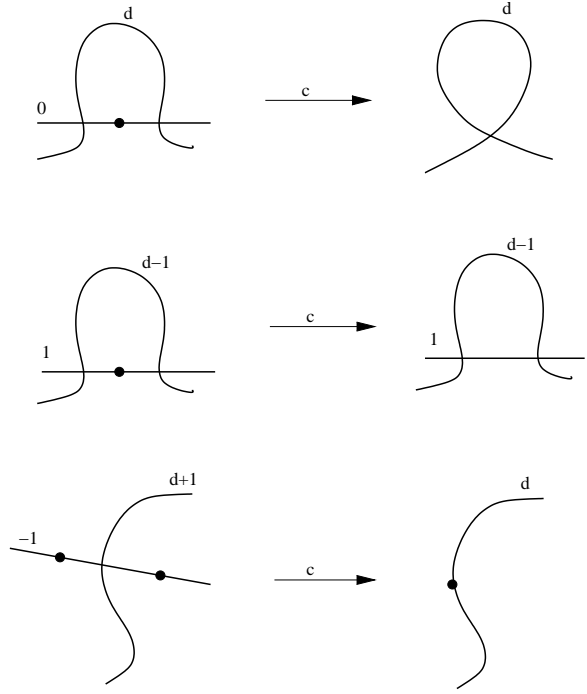}}
\end{center}
\caption{Contractions of quasistable pointed curves over $k$ and balanced degree $d$ line bundles.}
\label{excontraction}
\end{figure}

Now, suppose $r$ is a nodal point of $X$. Then, since $f(p_{n+1})=r=f(E)$, we must have that $k_E=2$ and that the only marked point of $X$ in $E$ is $p_{n+1}$ (otherwise, the condition that $f(p_i)=q_i$ for $i=1,\dots,n$ would imply that one of these $q_i$'s should coincide with $q_{n+1}$, which is nodal, and $(X',q_i)$ would not be a pointed curve).
We must do a further distinction here. Suppose first that $E$ intersects just one irreducible component of $X$: call it $F$ and $F'$ its associated irreducible component on $X'$. Now, if $X=E\cup F$, and if $F$ is rational, $X'$ is an irreducible genus $1$ curve, which is clearly quasistable. If, instead, $g_{F}>0$ or if $k_F\geq 3$, we see that all destabilizing and exceptional components of $X'$ correspond to destabilizing and exceptional components of $X$ and are contained in the same type of rational chains.

If, instead, $E$ intersects two distinct irreducible components of $X$, it is easy to see that, also in this case, all destabilizing and exceptional components of $X'$ correspond to destabilizing and exceptional components of $X$ and are contained in the same type of rational chains. So, $(X',q_i)$ will be quasistable if $(X,p_i)$ is.

Now, since $p_{n+1}$ is the only marked point in $E$, all irreducible components of $X'$ have the same marked points that the corresponding irreducible components of $X$, so
$f^*L'(q_1+\dots+q_n) \cong L(p_1+\dots + p_n)$ implies that the multidegree of $L$ on irreducible components of $X'$ coincides to the multidegree of $L$ on the corresponding irreducible components of $X$ and that the degree of $L$ on $E$ is zero. Let $Z'$ be a proper subcurve of $X'$ and $Z$ the corresponding proper subcurve of $X$. If $Z$ does not intersect $E$ or if it intersects $E$ in a single point, then it is immediate to see that inequality (\ref{basic}) holds for $L$ and $Z$ if and only if it holds for $L'$ and $Z'$. If, instead, $Z$ intersects $E$ in two points, then $g(Z')=g(Z)+1$, $t_{Z'}=t_Z$, $b_{Z'}^L=b_Z^L-1$ and $k_{Z'}=k_Z-2$, so, we get that
$$m_{Z'}(d,L)=m_{Z}(d,L)$$
and
$$m_{Z'}(d,L')=m_Z(d,L).$$
Since also $\deg_{Z'}L'=\deg_ZL$, we conclude that if we are contracting a rational bridge, if $L$ is balanced, also $L'$ will be balanced. Now, to conclude that the fact that $L'$ is balanced implies that also $L$ is balanced we have to further observe that, by definition of contraction, the degree of $L$ on $E$ is forced to be $0$ and that the inequality (\ref{basic}) is verified on $\o{X\sm E}$ (that does not correspond to any proper subcurve of $X'$), which is true since $m_{\o{X\sm E}}(d,L)=M_{\o{X\sm E}}(d,L)=d=\deg_{\o{X\sm E}}L$.

\end{proof}

The following lemma is Corollary 1.5 of \cite{knudsen}.

\begin{lem}\label{cor1.5}
Let $X$ and $Y$ be $S$-schemes and $f:X\rightarrow Y$ a proper $S$-morphism, whose fibers are at most one-dimensional. Let $\mathcal F$ be a coherent sheaf on $X$, flat over $S$ such that $H^1(f^{-1}(y)\mathcal F\otimes_{\mathcal O_Y}k(y))=(0)$ for each closed point $y\in Y$. Then $f_*\mathcal F$ is $S$-flat, $R^1f_*\mathcal F=0$ and, given any morphism $T\to S$, there is a canonical isomorphism
$$f_*\mathcal F\otimes_{\mathcal O_S}\mathcal O_T \cong (f\times1)_*(\mathcal F\otimes_{\mathcal O_S}\mathcal O_T).$$
If, moreover, $\mathcal F\otimes_{\mathcal O_Y}k(y)$ is globally generated we have also that the canonical map $f^*f_*\mathcal F\rightarrow \mathcal F$ is surjective.
\end{lem}

\begin{cor} \label{morebasicprop}
Let $(\pi:X\to S,s_i:S\to X,L)$ be an $(n+1)$-pointed quasistable curve endowed with a balanced line bundle of degree $d>>0$. Let $M$ be either the line bundle $L(s_1+\dots +s_n)$ or $(\omega_X(s_1+\dots+s_n))^3$. Then, for all $m\geq 1$, we have that
\begin{enumerate}
\item $\pi_*(M^m)$ is $S$-flat;
\item $R^1(\pi_*(M^m))=0$;
\item For all $i\geq 1$, the natural map
$$\alpha_i:\pi_*M^i\otimes \pi_*M\to\pi_*M^{i+1}$$
is surjective;
\item $\pi^*\pi_* M^m\to M^m$ is surjective.
\end{enumerate}
\end{cor}

\begin{proof}
(1) (2) and (4) follow immediately from Lemma \ref{cor1.5} and Corollaries \ref{basicprop} with $k=0$ and \ref{basicdualizing}, which assert that we can apply Lemma \ref{cor1.5} to $\pi$ and $M$, in both cases.

Let us now show that (3) holds. From Propositions \ref{normgenprop} and \ref{normgendualizing}, the statement holds if $S=\spec k$. Since $M$ satisfies the hypothesis of Lemma \ref{cor1.5}, the formation of $\pi_*$ commutes with base change. So, $\alpha_i$ is surjective at every geometric point of $S$ and we use Nakayama's Lemma to conclude that $\alpha_i$ is surjective.
\end{proof}

We now show that Knudsen's main lemma also holds for quasistable pointed curves and balanced line bundles of high degree.

\begin{lem} \label{mainlemma}
Let $d>>0$ and consider a contraction $f:X\rightarrow X'$ as in Definition \ref{contraction}. Denote by $M$ and $M'$, respectively, the line bundles $L(s_1+\dots +s_n)$ and $L'(t_1+\dots +t_n)$. Then, for all $m\ge 1$, we have that
\begin{enumerate}
 \item ${f^*(M')}^{m}\cong M^{m}$ and $(M')^{m}\cong f_*(M^{m})$;
\item $R^1f_*(M^{m})=0;$
\item $R^i\pi_*(M^{m})\cong R^i\pi'_*(M^{m})$ for $i\geq 0$.
\end{enumerate}

\end{lem}
\begin{proof}
That ${f^*(M')}^{m}$ is isomorphic to $M^{m}$ comes from our definition of contraction morphism. So, also $f_*f^*({M'}^{m})$ is isomorphic to $f_*(M^{m})$. So, composing this with the canonical map from ${M'}^{m}$ into $f_*f^*({M'}^{m})$, we get a map
$${M'}^{m}\rightarrow f_*(M^{m}).$$
Since the fibers of $f$ are at most smooth rational curves and $M$ is trivial on them, also $M^{m}$ is trivial on the fibers of $f$, so we can apply Lemma \ref{cor1.5} to it. Since the previous morphism is an isomorphism on the geometric fibers of $f$ and $(f_*M)^{m}$ is flat over $S$, we conclude that it is an isomorphism over $S$.

That $R^1f_*(M^{m})=0$ follows directly from Lemma \ref{cor1.5} while $(3)$ follows from $(1)$ and the Leray spectral sequence, which is degenerate by $(2)$.
\end{proof}

\subsection{Construction of the contraction functor}\label{const}

From now on, consider $d>>0$. Using the contraction morphism defined above, we will try to define a natural transformation from $\o\P_{d,g,n+1}$ to $\mathcal Z_{d,g,n}$.
Let $(\pi:X\rightarrow S,s_i:S\rightarrow X,L)$ be an $(n+1)$-pointed quasistable curve with a balanced line bundle $L$ of relative degree $d$.
For $i\geq 0$, define
$$\mathcal S_i:=\pi_*(L(s_1+\dots +s_n)^{\otimes i})$$
Since we are considering $d>>0$, then, by Corollary \ref{morebasicprop}, $R^1(\mathcal S_i)=0$, so $\mathcal S_i$ is locally free of rank $h^0(L(s_1+\dots +s_n)^{\otimes i})=i(d+n)-g+1$, for $i\geq 1$, .
Consider
$$\mathbb P(\mathcal S_1)\rightarrow S.$$
Again by Corollary \ref{morebasicprop}, the natural map
$$\pi^*(\pi_*L(s_1+\dots +s_n))\rightarrow L(s_1+\dots + s_n)$$
is surjective, so we get a natural $S$-morphism
\begin{equation*}
\xymatrix{
{X} \ar[r]^{q} \ar[d]^{\pi} & {\mathbb P(\mathcal S_1)} \ar[dl]\\
{S} \ar @{->} @/^/[u]^{s_i}
}
\end{equation*}
Define $Y:=q(X)$, $N:=\mathcal O_{\mathbb P(\mathcal S_1)}(1)_{|Y}$, and, by abuse of notation, call $q$ the (surjective) $S$-morphism from $X$ to $Y$. $N$ is an invertible sheaf over $Y$ and $q^*N \cong L(s_1+\dots +s_n)$.

Moreover, by Corollary \ref{morebasicprop} (3), we have that
$$Y\cong \mathcal Proj(\oplus_{i\geq 0}\mathcal S_i).$$ So, since all $\mathcal S_i$ are flat over $S$ (again by Corollary \ref{morebasicprop}), also $Y$ is flat over $S$, so it is a projective curve over $S$ of genus $g$ (since the only possible contractions are of rational components).

So, if we endow $\pi_c:Y\rightarrow S$ with the sections $t_i:=qs_i$, for $1\leq i\leq n$, the extra section $\Delta:=qs_{n+1}$ and $L^c:=N(-t_1-\dots -t_n)$ as above, we easily conclude that $q:X\rightarrow Y$ is a contraction.
Now, consider a morphism
\begin{equation*}
\xymatrix{
{ X} \ar[d]^{\pi} \ar[r]^{\beta_2}  & { X'} \ar[d]_{\pi'} \\
{S} \ar[r]_{\beta_1} \ar @{->} @/^/[u]^{s_i} & {S'} \ar @{->} @/_/[u]_{s_i'}
}
\end{equation*}
of $(n+1)$-pointed quasistable curves with balanced line bundles $L$ and $L'$ of relative degree $d$ in $\o\P_{d,g,n}$ and let us see that $(\beta_1,\beta_2,\beta_3)$, where $\beta_3$ is the isomorphism between $L$ and $\beta_2^*L'$, induces in a canonical way a morphism in $\mathcal Z_{d,g,n}$ between the contracted curves.

Define $\mathcal S':=\pi'_*L'(s_1' +\dots +s_n')$. Recall that, to give an $S'$-morphism from $\mathbb P(\mathcal S_1)$ to $\mathbb P(\mathcal S')$ is equivalent to give a line bundle $M$ on $\mathbb P(\mathcal S_1)$ and a surjection
$$(\beta_1\pi^c)^*(\pi'_*L')\rightarrow M$$
where by $\pi^c$ we denote the natural morphism $\mathbb P(\mathcal S_1)\rightarrow S$.
\begin{equation*}
\xymatrix{
{\mathbb P(\mathcal S_1)} \ar[dr]_{\pi^c} \ar @{-->} @/^18pt/[rrr] &{ X} \ar[l]_{q} \ar[d]^{\pi} \ar[r]^{\beta_2}  & { X'} \ar[r]^{q'} \ar[d]_{\pi'} & {\mathbb P(\mathcal S')} \ar[dl]\\
&{S} \ar[r]_{\beta_1} \ar @{->} @/^/[u]^{s_i} & {S'} \ar @{->} @/_/[u]_{s_i'}
}
\end{equation*}
Since we are considering $d>>0$, for $s'\in S'$, $h^0((\pi')^{-1}(s'), L'_{|(\pi')^{-1}(s')})$ is constant and equal to $d+n-g+1$. So, we can apply the theorem of cohomology and base change to conclude that there is a natural isomorphism
$$\beta_1^*\pi'_*L'\cong\pi_*\beta_2^*L'.$$
Since we have also the isomorphism
$$\beta_3:L\rightarrow \beta_2^*L'$$
we get a natural isomorphism
$$\pi^{c*}\beta_1^*(\pi'_*L')\cong{\pi^c}^*(\pi_*L).$$
Composing this with the natural surjection
$${\pi^c}^*(\pi_*L)\rightarrow \mathcal O_{\mathbb P(\mathcal S_1)}(1),$$
we conclude that there is a canonical surjection
$${\pi^c}^*\beta_1^*(\pi'_*L')\rightarrow \mathcal O_{\mathbb P(\mathcal S_1)}(1)$$
defining a natural $S'$-morphism from $\mathbb P(\mathcal S_1)\rightarrow \mathbb P(\mathcal S')$. This morphism naturally determines a morphism from $X^c$ to ${X'}^c$, where ${X'}^c$ is the image of $X'$ in $\mathbb P(\mathcal S')$ via $q'$, inducing a natural isomorphism between $L^c$ and the pullback ${L'}^c$, the restriction of $\mathcal O_{\mathbb P(\mathcal S')}(1)$ to ${X'}^c$. The fact that all these morphisms are canonical implies that this construction is compatible with the composition of morphisms, defining a natural transformation. We have just proved the following proposition.

\begin{prop}
There is a natural transformation $c$ from $\o\P_{d,g,n+1}$ to $\mathcal Z_{d,g,n}$ given on objects by the contraction morphism defined above.
\end{prop}

\subsection{Proof of the main Theorem}\label{proof}

We can now prove our main Theorem.

\begin{proof}(of Theorem \ref{main})
We must show that the contraction functor is an equivalence of categories, i. e., it is fully faithful and essentially surjective on objects. The fact that it is full is immediate. We can also conclude easily that it is faithful from the fact that a morphism of $\mathbb P^1$ fixing 3 distinct points is necessarily the identity. In fact, contraction morphisms induce isomorphisms on the geometric fibers away from contracted components and the contracted components have at least 3 special points and it is enough to use flatness to conclude.

In order to show that $c$ is essentially surjective on objects we will use Knudsen's stabilization morphism (see \cite{knudsen}, Def. 2.3) and check that it works also for pointed quasistable curves with balanced line bundles.

So, let $\pi:X\rightarrow S$ be an pointed quasistable curve, with $n$ sections $s_1,\dots,s_n$, an extra section $\Delta$ and a balanced line bundle $L$ on $X$, of relative degree $d$. Let $\mathcal I$ be the $\mathcal O_X$-ideal defining $\Delta$. Define the sheaf $\mathcal K$ on $X$ via the exact sequence
$$0\rightarrow \mathcal O_X\stackrel{\delta}{\rightarrow}\check{\mathcal I}\oplus(s_1+\dots s_n)\rightarrow \mathcal K\rightarrow 0$$
where $\delta$ is the diagonal morphism, $\delta (t)=(t,t)$.

Define
$$X^s:=\mathbb P(\mathcal K).$$
and let $p:X_s\rightarrow X$ be the natural morphism from $X^s$ to $X$. Theorem 2.4 of \cite{knudsen} asserts that, in the case that $X$ is a pointed stable curve, the sections $s_1\dots,s_n$ and $\Delta$ have unique liftings $s_1',\dots,s_{n+1}'$ to $X^s$ making $X^s\rightarrow S$ an $(n+1)$-pointed stable curve and $p:X^s\rightarrow X$ a contraction. One checks easily that the same construction holds also if $X$ is a quasistable pointed curve instead of a stable one. In fact, the assertion is local on $S$, the problem being the points where $\Delta$ meets non-smooth points of the fibre or other sections since in the other points $X^s$ is isomorphic to $X$. In the case where $\Delta$ meets a non-smooth point of a geometric fiber, locally $X^s$ is the total transform of the blow-up of $X$ at that point with the reduced structure and $s_{n+1}'$ is a smooth point of the exceptional component. In the case where $\Delta$ coincides with another section $s_i$ in a geometric fiber $X_s$ of $X$, then, locally, on $X^s$ is the total transform of the blow-up of $X$ at $s_i(s)$, again with the reduced structure, and $s_{i}'$ and $s_{n+1}'$ are two distinct smooth points of the exceptional component.

\begin{figure}
\begin{center}
\scalebox{.6}{\epsfig{file=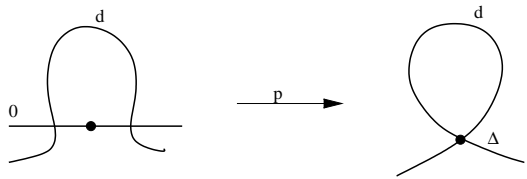}}
\end{center}
\caption{Stabilization of pointed quasistable curves with balanced degree $d$ line bundles.}
\label{stabex}
\end{figure}

Let $L^s:=p^*L$. Then the multidegree of $L^s$ on a geometric fiber $X_s^s$ coincides with the multidegree of $L$ in the irreducible components of $X_s^s$ that correspond to irreducible components of $X$ and, in the possibly new rational components, the degree is $0$. So, $L^s$ is balanced of relative degree $d$.

To conclude, we must check that $c(X^s)$ is isomorphic to $X$. By definition, $c(X^s)$ is given by the image of $X^s$ on $\mathbb P (\pi^s_*(L^s(s_1'+\dots +s_n')))$. Consider the line bundle $L(s_1+\dots +s_n)$ on $X$. By Corollary \ref{morebasicprop}, there is a natural surjection
$$\pi^*\pi_*(L(s_1+\dots + s_n))\rightarrow L(s_1+\dots +s_n).$$
But, since $\pi_*(L(s_1+\dots +s_n))$ is naturally isomorphic to $\pi_*^sp^*L(s_1+\dots +s_n)$, we get a natural surjection
$$\pi^*(\pi_*^sL^s(s_1'+\dots +s_n'))\rightarrow L(s_1+\dots +s_n)$$
so, equivalently, a morphism $f$ from $X$ to $\mathbb P (L^s(s_1'+\dots +s_n'))$. Since $p^*(L(s_1+\dots + s_n))=L^s(s_1'+\dots + s_n')$ induces the natural morphism $q:X^s\rightarrow \mathbb P(\pi_*^s(s_1' +\dots +s_n'))$, whose image is $c(X^s)$, naturally the image of $f$ is $c(X^s)$. It is easy to check that $f$ is an isomorphism on the geometric fibers, so, by flatness, we conclude that $f$ gives an $S$-isomorphism between $X$ and $c(X^s)$ as pointed quasistable curves and determines an isomorphism between the respective balanced degree $d$ line bundles.
\end{proof}

\section{The forgetful morphism from $\o\P_{d,g,n}$ onto $\o{\mathcal M}_{g,n}$}\label{forgetful}

Now, for each $n>0$, we will try to construct a morphism $\Psi_{d,g,n}:\o\P_{d,g,n+1}\to \o{\M}_{g,n}$ fitting in diagram (\ref{induction}) above.

Let $(\pi:X\to S,s_i:S\to X)$, $i=1,\dots,n$ be an $n$-pointed quasistable curve over $S$. Denote by $\omega$ the line bundle $\omega_{X/S}^3(s_1+\dots+s_n)$. Then, by Corollary \ref{morebasicprop}, $R^1(\pi_*\omega)=0$, so it is locally free and there is an $S$-morphism $\gamma:X\to \mathbb P(\omega)$ making the following diagram commute.

\begin{equation}
\xymatrix{
{X} \ar[rr]^{\gamma} \ar[dr]^{\pi} & & {\mathbb P(\pi_*(\omega))} \ar[dl] \\
& {S} \ar @{->} @/^/[ul]^{s_i} &
}
\end{equation}

The restriction of $\gamma$ to any fiber $X_s$ of $\pi$ maps $X_s$ to its stable model in $\mathbb P(\omega)$, which is naturally endowed with the sections $\gamma s_i$, for $i=1,\dots, n$. This follows from the fact that $\omega$ is very ample on the stable components of each fiber, whereas it has degree $0$ on the exceptional components. Moreover, $\gamma(X)$ is flat over $S$. In fact, from Corollary \ref{morebasicprop}, for any $i\geq 1$, the natural map
$$\pi_*\omega^i\otimes \pi_*\omega\to \pi_*\omega^{i+1}$$
is surjective. It follows that $\gamma(X)\cong \mathcal Proj(\oplus_{i\geq 0}\pi_*(\omega_X(s_1+\dots +s_n))^i)$, which is flat over $S$ because each $\pi_*(\omega_X(s_1+\dots+s_n))^i$ is $S$-flat, again by Corollary \ref{morebasicprop}.

It is easy to check that this yields a surjective morphism $\Psi_{d,g,n}$ from $\o\P_{d,g,n}$ onto $\o\M_{g,n}$ fitting in diagram (\ref{induction}) and making it commutative.

Let $(\pi:X\to S,s_i:S\to X,L)$ be an $(n+1)$-pointed quasistable curve of genus $g$ endowed with a balanced line bundle $L$ of relative degree $d$ over $X$. It is immediate to check that, restricting ourselves to the geometric fibers of $\pi$, the diagram is commutative since in both directions we get the $n$-pointed curve which is the stable model of the initial one, after forgetting the last point. Now, since all families are flat over $S$, we conclude that the diagram is commutative.

The surjectivity of $\Psi_{d,g,n}$ follows from the fact that $\Psi_{d,g,0}$ is surjective (see \cite{capneron}, Proposition 4.12) and from the commutativity of the diagram because $\Phi_{d,g,n}$ and $\Pi_{g,n}$ are the universal morphisms onto $\o\P_{d,g,n-1}$ and $\o{\M}_{g,n}$, respectively.

So, the fibers of $\Psi_{d,g,n}$ over a pointed curve $X'\in\o{\mathcal M}_{g,n}$ are the quasistable pointed curves $X$ with stable model $X'$ endowed with balanced degree $d$ line bundles.

\section{Further properties}\label{properties}

Let $X$ be an $n$-pointed quasistable curve over $k$. By applying the contraction morphism we get an $(n-1)$-pointed quasistable curve with an extra section. If we forget about this extra section and we iterate the contraction procedure $n$ times, at the end we get a quasistable curve with no marked points, call it $X_0$. Denote by $f$ this morphism from $X$ to $X_0$.

Let $\omega_{X_0}$ be the dualizing sheaf of $X_0$. For each proper subcurve $Z_0$ of $X_0$, the degree of $\omega_{X_0}$ in $Z_0$ is $w_{Z_0}=2g_{Z_0}-2+k_{Z_0}$. In particular, it has degree $0$ on exceptional components of $X_0$. Consider now the pullback of $\omega_{X_0}$ via $f$, $f^*(\omega_{X_0})$. This is a line bundle on $X$ having degree $0$ on rational bridges and on rational tails; moreover, given a  proper subcurve $Z$ of $X$ whose image under $f$ is a proper subcurve $Z_0$ of $X_0$, $f^*(\omega_{X_0})$ has degree $w_{Z_0}=w_{Z}-t_Z$ on $Z$.

So, a line bundle $L$ of degree $d$ on $X$ with given balanced multidegree on rational tails and rational bridges of $X$ is balanced on $X$ if and only if $L\otimes f^*(\omega_{X_0})$ is balanced on $X$ of degree $d+(2g-2)$ and with the same multidegree on rational tails and rational bridges. In fact, for each proper subcurve $Z$ of $X$ which is not contained in rational tails or rational bridges, we have that
$$\mbox{deg}_Z(L\otimes f^*(\omega_{X_0}))=\mbox{deg}_{Z}L+w_Z-t_Z\le $$
$$\le \frac{dw_Z}{2g-2}+\frac{g-1-d}{2g-2}t_Z-b_Z^L+\frac{k_Z}{2}+w_Z-t_Z=$$
$$=\frac{(d+2g-2)w_Z}{2g-2}+\frac{g-1-(d+2g-2)}{2g-2}t_Z-b_Z^L+\frac{k_Z}{2}$$
and similarly that
$$\mbox{deg}_Z(L\otimes f^*(\omega_{X_0}))\ge \frac{dw_Z}{2g-2}+\frac{3g-3-d}{2g-2}t_Z-b_Z^L-\frac{k_Z}{2}+w_Z-t_Z=$$
$$=\frac{(d+2g-2)w_Z}{2g-2}+\frac{3g-3-(d+2g-2)}{2g-2}t_Z-b_Z^L-\frac{k_Z}{2},$$
so $(L\otimes f^*(\omega_{X_0}))_{|Z}$ satisfies inequality (\ref{basic}) if and only if $L_{|Z}$ does.

In conclusion, we have the following result.

\begin{prop} Let $d$ and $d'$ be integers such that there exists an $m\in \Z$ such that $d'=d+m(2g-2)$. Then, $\o{\mathcal P}_{d,g,n}$ and $\o{\mathcal P}_{d',g,n}$ are isomorphic.
\end{prop}

\begin{proof}
We must show that there is an equivalence of categories between $\o{\mathcal P}_{d,g,n}$ and $\o{\mathcal P}_{d',g,n}$. So, let $(\pi:X\to S,s_i:S\to X, L)$, $i=1,\dots,n$ be an object of $\o{\mathcal P}_{d,g,n}$. Consider its image under $\Phi_{d,g,0}\circ\Phi_{d,g,1}\circ\dots\circ\Phi_{d,g,n}$ and denote it by $(\pi_0:X_0\to S,L_0)$. According to \ref{const}, there is an $S$-morphism $q_0:X\to X_0$. Then, by what we have seen above, $L':=L\otimes q_0^*(\omega_{X_0/S}^m)$ is a balanced line bundle of relative degree $d'$ over $X$, so $(\pi:X\to S,s_i:S\to X,L')\in\o{\mathcal P}_{d',g,n}$.

It is easy to check that this defines an equivalence between $\o{\mathcal P}_{d,g,n}$ and $\o{\mathcal P}_{d',g,n}$.
\end{proof}

\begin{prop}
For all $n>0$, there are forgetful morphisms $\Phi_{d,g,n}:\o{\mathcal P}_{d,g,n+1}\rightarrow \o{\mathcal P}_{d,g,n}$ endowed with $n$ sections $\sigma^1_{d,g,n},\dots,\sigma^{n}_{d,g,n}$ yielding Cartier divisors $\Delta^i_{d,g,n+1}$, $i=1,\dots,n$ such that $\sigma^i_{d,g,n}$ gives an isomorphism between $\o{\mathcal P}_{d,g,n}$ and $\Delta^i_{d,g,n+1}$.
\end{prop}

\subsection{Rigidified balanced Picard stacks}

Note that each object $(\pi:X\to S,s_i:S\to X, L)$, $i=1,\dots,n$, in $\o{\mathcal P}_{d,g,n}$ has automorphisms given by scalar multiplication by an element of $\Gamma ( X,\mathbb G_m)$ along the fibers of $ L$. Since these automorphisms fix $ X$, there is no hope $\o{\mathcal P}_{d,g,n}$ can be re\-pre\-sen\-ta\-ble over $\o{\mathcal M}_{g,n}$ (see \cite{av}, 4.4.3). The rigidification procedure, defined in \cite{avc}, fits exactly on our set up and produces an algebraic stack with those automorphisms removed.

Denote by $\o\P_{d,g,n}\fatslash \mathbb G_m$ the rigidification of $\o\P_{d,g,n}$ along the action of $\mathbb G_m$.
Exactly because the action of $\mathbb G_m$ on $\o{\mathcal P}_{d,g,n}$ leaves $\o{\mathcal M}_{g,n}$ invariant, the morphism $\Psi_{d,g,n}$ descends to a morphism from $\o{\mathcal P}_{d,g,n}\fatslash \mathbb G_m$ onto $\o{\mathcal M}_{g,n}$, which we will denote again by $\Psi_{d,g,n}$, making the following diagram commutative.
\begin{equation}\label{induction2}
\xymatrix{
& {\o{\mathcal P}_{d,g,n}\fatslash \mathbb G_m} \ar[rd]^{\Psi_{d,g,n}} \ar[dl]_{\Phi_{d,g,n}} & \\
{\o{\mathcal P}_{d,g,n-1}} \ar[dr]_{\Psi_{d,g,n-1}} & & {\o{\mathcal M}_{g,n}} \ar[dl]^{\Pi_{g,n}}\\
& {\o{\mathcal M}_{g,n-1}} &
}
\end{equation}

So, the same argument we used to show that $\Psi_{d,g,n}$ is universally closed for all $n\geq 0$ if and only if $\Psi_{d,g,0}$ is universally closed holds also in this case. Moreover, since, for $n=0$, we have that $\Psi_{d,g,0}:[H_d/G]\to \o{\mathcal M}_g$ is proper and strongly representable if and only if $(d-g+1,2g-2)=1$, we have that the same statement holds in general for every $n\geq 0$.

\begin{prop}
Let  $\o{\mathcal P}_{d,g,n}\fatslash \mathbb G_m:=\o{\mathcal P}_{d,g,n}\fatslash\mathbb G_m$ be the rigidification of $\o{\mathcal P}_{d,g,n}$ under the action of $\mathbb G_m$. Then, if $(d-g+1,2g-2)=1$, $\o{\mathcal P}_{d,g,n}\fatslash \mathbb G_m$ is a Deligne-Mumford stack of dimension $4g-3+n$ with a proper and strongly representable morphism onto $\o{\mathcal M}_{g,n}$.
\end{prop}

\section{Acknowledgements}
This paper is part of my Ph.D. thesis. I am very grateful to my advisor Professor Lucia Caporaso for the patience with which she followed this work very closely and for always transmitting me lots of encouragement and mathematical enthusiasm.
I am also grateful to
Simone Busonero and Filippo Viviani for giving me important suggestions.

\bigskip

\noindent Margarida Melo\\ Dipartimento di Matematica, Universit\`a di Roma Tre - \\
Departamento de Matem\'atica, Universidade de Coimbra.\\
E-mail: \textsl{melo@mat.uniroma3.it}, \textsl {mmelo@mat.uc.pt}.

\end{document}